\documentclass[10pt]{elsarticle}

\usepackage{amsmath,amsfonts,amssymb,amsthm,bbold,bbm}

\def\uno{\mathbb 1}
\def\vol{\text{vol }}
\def\t{\text{\itshape{\textsf{T}}}} 

\def\ep{\varepsilon}
\def\bar{\overline}
\def\fro{\mathrm{F}}

\def\diag{W} 

\newtheorem{theorem}{Theorem}[section]
\newtheorem{lemma}[theorem]{Lemma}
\newtheorem{corollary}[theorem]{Corollary}
\newtheorem{definition}[theorem]{Definition}
\newtheorem{remark}[theorem]{Remark}

\makeatletter
\def\ps@pprintTitle{%
	\let\@oddhead\@empty
	\let\@evenhead\@empty
	\def\@oddfoot{}%
	\let\@evenfoot\@oddfoot}
\makeatother

\begin{document}

\title{Generalized modularity matrices}

\author[dcfa]{Dario Fasino\fnref{fndf}} 
\ead{dario.fasino@uniud.it}

\address[dcfa]{Department of Chemistry, Physics, and Environment,
University of Udine, Udine, Italy.}

\author[sb]{Francesco Tudisco\fnref{fnft}}
\ead{tudisco@cs.uni-saarland.de}

\address[sb]{Department of Mathematics and Computer Science, Saarland University, Saarbr\"ucken, Germany.}
\fntext[fndf]{The work of this author has been partially supported by INDAM-GNCS.}
\fntext[fnft]{The work of this author has been partially supported by the ERC Grant NOLEPRO.}
\begin{keyword}
Community detection, modularity matrix, nodal domains.

\MSC 05C50, 15A18, 15B99
\end{keyword}

\begin{abstract}
Various modularity matrices appeared in the recent literature
on network analysis and algebraic graph theory. Their purpose
is to allow writing as quadratic forms certain combinatorial functions appearing in the framework of graph clustering problems. In this paper we put in evidence certain common traits of various 
modularity matrices and shed light on their spectral properties that are at the basis of various theoretical results and practical spectral-type algorithms for community detection.
\end{abstract}

\maketitle

\section{Introduction}

Consider the following problem: 
We have a group of individuals, objects, or documents,
bound together by a kind of reciprocal similarity relationship,
and we want to localize a cluster, a tightly knit subset of such group that can be recognized as a ``community'',
in some sense.
In the common terminology of network science, this is an example of a community detection problem \cite{santo-fortunato,Schaeffer2007}.
In fact, community detection problems are among the most 
relevant problems in the analysis of complex networks.

Networks are widely used to model a large variety of real life systems and appear in many fields of scientific interests. Community detection
and graph clustering methods may reveal many significant 
network properties and, as a consequence, are receiving a considerable amount of attention from various research areas, see e.g., 
\cite{Louvain_method, Estrada2011, GA05}. 
One of the most popular method for community detection is that of \textit{modularity}. The idea was proposed by Newman and Girvan in \cite{newman-girvan} and is essentially based on the maximization of a function called indeed modularity. However there is no clear or universally accepted definition of community in a graph; despite of this,  almost any recent definition or  community detection method is based on the maximization of a quadratic quality function related with the original modularity, see for instance, \cite{multires-2,multires-1,Ronhovde_Nussinov}. 

In this paper we basically propose a unified framework for a number of modularity-type matrices and functions borrowed from recent literature on community detection, and we analyse their spectral properties that are 
of possible interest for community detection methods.
In particular, we prove a modularity-oriented version of a well known theorem due to Fiedler \cite[Thm.\ 3.3]{fiedler-vector} that holds for the Laplacian matrix of a graph.
Our theorem holds for any negative semidefinite rank-one perturbation of a  
symmetric matrix $A$ with nonnegative off diagonal entries, and can be used to ensure the connectivity of the modules generated by the best known algorithms for community detection inspired by the renowned spectral partitioning method.

This paper is organized as follows.
After introducing hereafter our main notation, in Section \ref{sec:motivations} we present briefly a number of topics arising in graph clustering literature, which provide several relevant examples where our concept of
generalized modularity matrix comes from. 
In the subsequent Section \ref{sec:main} we prove our main result, which shows that a certain nodal domain of a leading eigenvector of a generalized modularity matrix is connected. 
In the successive sections we deepen the study of spectral properties of generalized modularity matrices. In fact, we consider the identifiability of a prescribed cluster
as a nodal domain of the leading eigenvector (Section 4), the increase of the largest eigenvalue due to a newly added edge (Section 5), and the relationship between positive eigenvalues of a modularity matrix and the number of distinct clusters that can be recognized in a given network (Section 6). 
Finally, Section \ref{sec:conclusions} is used to point out some conclusive remarks.

\subsection{Notations and preliminaries}\label{sec:notations}

A {\em symmetric weighted graph} $G$ is a pair $(V,E)$ where $V$ is a finite set of nodes (or vertices), 
and $E : V\times V\mapsto \mathbbm{R}_{\geq 0}$ is a nonnegative weight function defined over edges, that is, node pairs, where $E(i,j) = E(j,i)$. In practice, edges with larger weights represent stronger connections among nodes, so missing edges get weight $0$.
If $E(i,i) > 0$ then we have a {\em loop} on node $i$.
Any graph considered in the following is assumed symmetric, weighted, and connected.
Since $V$ is finite we freely identify it with $\{1,\dots,n\}$.

There exists a natural bijection that associates to any graph $G$ a componentwise nonnegative, irreducible, symmetric matrix $A\equiv(a_{ij})$, called adjacency matrix, defined by 
$a_{ij} = E(i,j)$.
Further relevant notation is listed below.
\begin{itemize}

\item For any $i\in V$, $d_i$ denotes its degree, 
$d_i=\sum_{j\in V} a_{ij}$. 
The vector of degrees of $G$ is denoted by $d = (d_1,\ldots,d_n)^\t$.

\item For any $S\subseteq V$ we denote by
$\bar S$ the complement $V \setminus S$ and let $\vol S = \sum_{i \in S}d_i$ be the {\em volume} of $S$. 
In particular,     
$\vol G = \sum_{i \in V}d_i$ is the volume of the whole graph.

\item
For any $S\subseteq V$, if $X$ is an $n\times n$ matrix then we denote by $X(S)$ the principal submatrix of $X$ whose indices are in $S$.
Analogously, we denote by $G(S)$ the subgraph of $G$ induced by nodes in $V$, that is the graph whose
adjacency matrix is $A(S)$.

\item Let $\uno$ denote the vector of all ones whose dimension depends on the context. Furthermore, 
for any $S \subseteq \{1,\dots,n\}$ we let $\uno_S$ be its \textit{characteristic vector}, defined as $(\uno_S)_i =1$ if $i \in S$ and $(\uno_S)_i=0$ otherwise.

\item The cardinality of a set $S$ is denoted by $|S|$. 
In particular, $|V|=n$.

\item For a matrix $A$ and a vector $x$, we write $A\geq O$ or $x\geq 0$ (resp. $A>O$ or $x>0$) to denote  componentwise nonnegativity (resp., positivity). 

\item If $X$ is a symmetric matrix then its eigenvalues are
denoted by $\lambda_i(X)$ and are ordered as $\lambda_1(X)\geq \cdots\geq \lambda_n(X)$, unless otherwise specified.
\end{itemize}

We will freely use familiar properties of matrices
such as the variational characterization of eigenvalues of symmetric matrices, Gershgorin's eigenvalue localization theorem,
and fundamental results in
Perron--Frobenius theory, see e.g., 
\cite{positive-book,wilkinson}. 
For completeness, we recall hereafter 
some important facts concerning the symmetric eigenvalue problem:
\begin{itemize}
\item (Cauchy interlacing theorem)
Let $A\in\mathbbm{R}^{n\times n}$ be a symmetric matrix 
and let $Z\in\mathbbm{R}^{n\times (n-k)}$ be a matrix 
with orthonormal columns.
Then, for all $i = 1,\ldots,n-k$,
\begin{equation}   \label{eq:interlacing}
   \lambda_i(A) \geq \lambda_i(Z^\t AZ) \geq \lambda_{i+k}(A) .
\end{equation}
\item 
Let $A\in\mathbbm{R}^{n\times n}$ be a symmetric matrix 
and let $B\in\mathbbm{R}^{(n-k)\times (n-k)}$ 
be a principal submatrix of $A$. Then, for all $i = 1,\ldots,n-k$,
\begin{equation}   \label{eq:interlacing2}
   \lambda_i(A) \geq \lambda_i(B) \geq \lambda_{i+k}(A) .
\end{equation}
\item
(Weyl's inequalities)
Let $A$ be a real symmetric matrix of order $n$
and $v\in\mathbbm{R}^n$. Then, for $i = 1,\ldots,n-1$,
\begin{equation}   \label{eq:Weyl}
   \lambda_i(A) \geq \lambda_{i+1}(A+vv^\t)
   \geq \lambda_{i+1}(A) .
\end{equation}
\end{itemize}

\section{Motivations and overview}\label{sec:motivations}

The discover and description of communities in a graph is a central problem in modern graph analysis; 
an elementary overview of graph clustering problems and techniques is the survey \cite{Schaeffer2007}.
 Although intuition suggests that a community (or cluster) in $G$ should be a possibly connected group of nodes 
whose internal connections are stronger than  
those with the rest of the network,
there is no universally accepted definition of community. 
A survey of several proposed definitions of community can be found in \cite{santo-fortunato}.
However,
as the author of that paper therein underlines, the 
definition based on the modularity quality function is by far the most popular one. The modularity function was proposed by Newman and Girvan in \cite{newman-girvan} as a possible measure of whether a subgraph 
of $G$ is a cluster or not. They assert that a subset $S\subseteq V$ is a cluster 
if the induced subgraph $G(S)$ contains more edges than 
those expected if edges were placed at random
preserving node degrees. 
All such subsets are indeed those having positive modularity. Since no information on the connectedness nor the dimension of the clusters is  given by subsets with positive modularity, we shall call such subgraphs not just communities but rather \textit{modules}. Let us formalize such concept. Consider a graph $G$ and the associated adjacency matrix $A$. The graph $G$ may have loops, and edges may be weighted, so that $A$ is a rather arbitrary nonnegative matrix. If $d=A\uno$ is the degree vector and $\vol G = \sum_i d_i$ is the volume of the graph, the 
\textit{Newman--Girvan modularity matrix} of $G$ is 
defined as \cite{newman-eigenvectors,newman-modularity,newman-girvan}
\begin{equation}\label{eq:modularity-matrix}
   M_{\mathrm{NG}} = A - \frac{1}{\vol G}dd^\t
\end{equation}
and the modularity measure of a subset $S\subseteq V$ is 
usually given by the associated quadratic form
$$ 
   Q_{\mathrm{NG}}(S)=\uno_S^\t M_{\mathrm{NG}} \uno_S
$$ 
where $\uno_S$ denotes the characteristic vector of the set $S\subseteq V$. Thus modules are subgraphs $G(S)$ such that $Q_{\mathrm{NG}}(S)>0$. A module which is connected and has a considerable size is commonly considered as a good community candidate. Remark the equivalent formulas
$$
   Q_{\mathrm{NG}}(S)=
   \uno_S^\t A \uno_S - \frac{(d^\t\uno_S)^2}{\vol G} = 
   e_{\mathrm{in}}(S) - \frac{(\vol S)^2}{\vol G} ,
$$
where 
\begin{equation}   \label{eq:e_in}
   e_{\mathrm{in}}(S) = \uno_S^\t A \uno_S
   = \sum_{i,j\in S} E(i,j) 
\end{equation}
is the overall strength of internal links.

Besides the Newman--Girvan matrix, several generalized modularity matrices appear in the community detection literature, often in a rather hidden form. 
Indeed, in \cite{FT14} we initially focused our investigations on $M_{\mathrm{NG}}$, but we realized 
afterward that a clear common structure is shared by a number of 
different modularity measures and
matrices appearing in this scientific area.
Thus we propose here a spectral analysis which uncovers common properties shared by all of them.    A generalized modularity matrix is any negative semidefinite rank-one correction of a real symmetric matrix with nonnegative off-diagonal entries. We shall denote any such a matrix with the symbol $M$ and we shall state a  formal definition in the subsequent Section \ref{sec:main}.

In the remaining part of this section we shortly discuss various topics arising in the community detection literature, 
presenting other modularity-type matrices and motivating the introduction of generalized modularity matrices in turn.

\subsection{Newman's spectral method}\label{sec:SSGB}
 
A major task in community detection is to look for a module in $G$ having maximal modularity,
briefly called \textit{a leading module} in what follows. 
The probably best known methods for detecting a leading module are based on the idea of spectral techniques, firstly introduced
in graph partitioning problems.

Consider the set $\{0,1\}^n$ of $n$-dimensional vectors whose components are only $0$ or $1$. Clearly $Q_*=\max_{S\subseteq V}Q_{\mathrm{NG}}(S) = \max_{v \in \{0,1\}^n}v^\t M_{\mathrm{NG}} v$. Now let $u_1,\dots, u_n$ be the (real) orthonormal eigenvectors of $M_{\mathrm{NG}}$, then $M_{\mathrm{NG}} = \sum_i \lambda_i(M_{\mathrm{NG}}) u_iu_i^\t$ and 
$v^\t M_{\mathrm{NG}} v = \sum_i \lambda_i(M_{\mathrm{NG}}) (u_i^\t v)^2$. 
If $v$ could be chosen to be proportional to $u_1$ then 
the sum would be maximized.
However the constraint $v \in \{0,1\}^n$ prevents us to such a simple choice and makes the optimization problem much more difficult. In fact it has been pointed out in several works, as for instance  \cite{newman-eigenvectors,newman-modularity}, that it is  extremely unlikely that a simple procedure exists for finding the optimal $v \in \{0,1\}^n$. 
Spectral partitioning based methods essentially select $v$ accordingly with the sign of the elements in $u_1$, by setting $v_i =1$ if $(u_1)_i$ is positive (or nonnegative), and 
$v_i=0$ otherwise. Then the vertex set $V$  is partitioned into $P=\{i\in V\mid v_i=1\}$ and $N = \bar P$, and  $G(P)$ is proposed as an approximation of the module having maximal modularity in $G$. 

Although the described procedure proposes the subgraph $G(P)$ as a leading module, it can been shown that 
either $G(P)$ or $G(N)$ are connected subgraphs of $G$, depending on the orientation of $u_1$ \cite[Thm. 4.2]{FT14}. However (and unfortunately) if the sign of $u_1$ is chosen so that $G(P)$ is connected, it is not possible to ensure that $G(N)$ is connected as well, at least in the general case. Counterexamples are given in \cite{FT14}
and in the subsequent Remark \ref{rem:star}.

The described procedure provides a reasonably good bipartition of $G$. Typical networks, however, require a division into more than two parts, so a natural extension of the spectral method described so far has been proposed. Such idea was probably introduced by Newman in \cite{newman-modularity} and is at the basis of most of the modern algorithms for communities detection, see e.g. the renowned Louvain method \cite{Louvain_method}. We call this procedure \textit{Successive Spectral Graph Bipartition} algorithm (SSGB) and we briefly sketch it hereafter.

The spectral method previously described is used to divide the network into two parts $P$ and $N$, so that $V = P \cup N$.  Then those parts are bipartitioned again into $P_1$, $N_1$, $P_2$, and $N_2$ so that $P = P_1\cup N_1$, $N =P_2 \cup N_2$, and so forth. The crucial step here is that each time the modularity matrix for the subgraphs $G(P_i)$ and $G(N_i)$ must be considered, and of course it can not be done by simply considering the principal submatrices 
$M_{\mathrm{NG}}(P_i)$ and $M_{\mathrm{NG}}(N_i)$ respectively, since the degrees of vertices in the subgraphs change when some edge is removed. Instead, for each subset $S\subseteq V$ and respective subgraph $G(S)$, a new modularity matrix  $M^{S}_{\mathrm{NG}}$ is defined by setting
\begin{equation}\label{eq:sub-modularity}
	\textstyle{M^{S}_{\mathrm{NG}} = M_{\mathrm{NG}}(S) -\left(D_{G(S)}-\frac{\vol S}{\vol G}D(S)\right)}
\end{equation}
where $D$ is the diagonal matrix of the degrees of original graph $G$, whereas $D_{G(S)}$ is  the diagonal matrix of the degrees of the considered subgraph $G(S)$. The SSGB procedure stops when the computed modularity matrix $M^{S}_{\mathrm{NG}}$ has no positive eigenvalues. 
It is worth noting that already this very crucial  procedure   generates matrices whose structure is 
quite different with respect  the structure of $M_{\mathrm{NG}}$, due to the diagonal term. Thus the connectedness of subgraphs it produces is not  ensured anymore. However, all the matrices therein considered are generalized modularity matrices, as we will better discuss throughout the end of Section \ref{sec:main}. 

\subsection{A normalized variant of $M_{\mathrm{NG}}$}
\label{sec:normalized-M}

Let $D = \mathrm{Diag}(d_1, \dots, d_n)$ be the diagonal matrix of the degrees of the graph $G$. In analogy with the renowned normalized Laplacian matrix of a graph
\cite{chung}, 
the  \textit{normalized} version of the Newman--Girvan  modularity matrix is defined by
$$
   M_{\mathrm{norm}} = D^{-1/2} M_{\mathrm{NG}} D^{-1/2}\, . 
$$
Even though that matrix is not very popular in the community detection literature, 
$M_{\mathrm{norm}}$ appears in various networks related questions as the analysis of quasi-randomness properties of graphs with given degree sequences,
see e.g., \cite{MR2371048} or \cite[Chap.\ 5]{chung}. 
It is straightforward to see that the modularity measure induced by $M_{\mathrm{NG}}$ can also be defined as a quadratic form associated with $M_{\mathrm{norm}}$. In fact, if $v = D^{1/2}\uno_S $ then
$$
   Q_{\mathrm{NG}}(S) =
   \uno_S^\t M_{\mathrm{NG}} \uno_S = 
   v^\t M_{\mathrm{norm}} v\, .
$$
The effect of the diagonal scaling becomes apparent when
considering Rayleigh quotients instead of quadratic forms. 
Indeed, $\uno_S^\t M_{\mathrm{NG}} \uno_S/\uno_S^\t \uno_S = Q_{\mathrm{NG}}(S)/|S|$, whereas
$$
   \frac{v^\t M_{\mathrm{norm}} v}{v^\t v} = 
   \frac{\uno_S^\t M_{\mathrm{NG}} \uno_S}{\uno_S^\t D\uno_S} = 
   \frac{Q_{\mathrm{NG}}(S)}{\vol S}\,  .
$$ 
Note that  $A_{\mathrm{norm}} = D^{-1/2}AD^{-1/2}$ is a nonnegative irreducible matrix to which corresponds the symmetric weighted graph $G_{\mathrm{norm}}=(V,\hat E)$ whose weight function is $\hat E(i,j)=E(i,j)/\sqrt{d_id_j}$. 
Therefore  $M_{\mathrm{norm}}$ and $M_{\mathrm{NG}}$ share the crucial property 
of being a negative semidefinite rank-one correction of
the adjacency matrix of a graph.

\subsection{The matrix approach to the resolution limit}
\label{sec:res-limit}

Although modularity optimization techniques are very popular, recently it has been pointed out that they suffer a \textit{resolution limit},
see e.g., \cite{fortunato-res-limit-1, kumpula-res-limit, fortunato-res-limit-2}. In fact, it has been noted that  modularity maximization algorithms are inclined to merge small clusters into larger modules.
Various alternative modularity measures have been proposed in recent years, essentially based on the introduction of a 
tunable scaling coefficient $\gamma$ (usually called \textit{resolution parameter}) or on the insertion of weighted selfloops.

Starting from a statistical mechanics approach which interprets
community detection as finding the ground state of a spin system,
Reichardt and Bornholdt introduced in \cite{multires-1} 
a parametrized modularity measure for $S \subseteq V$.
In our notations, that definition reads
$$
   Q_\gamma(S) = e_{\mathrm{in}}(S) -
   (\gamma/\vol G) (\vol S)^2 ,
$$
where $\gamma>0$ is the resolution parameter and 
$e_{\mathrm{in}}(S)$ is as in \eqref{eq:e_in}.
We observe that, introducing the matrix
\begin{equation}   \label{eq:def_RB}
   M_{\mathrm{RB}} = A - (\gamma/\vol G) dd^\t ,
\end{equation}
then $Q_\gamma(S) = \uno_S^\t M_{\mathrm{RB}} \uno_S$,
and when $\gamma = 1$ then we recover Newman--Girvan modularity matrix \eqref{eq:modularity-matrix}.
Also from statistical mechanics considerations the 
parametrized modularity function
$$
  Q_\gamma(S) = e_{\mathrm{in}}(S) - \gamma |S|^2
$$
has been considered by Ronhovde and Nussinov in \cite{Ronhovde_Nussinov} as well as other authors, see e.g.,
\cite{RB2004,TVDN_CPM}, possibly with minor notational variations
or scaling factors.
By defining the matrix
$$
   M_{\mathrm{RN}} = A - \gamma\uno\uno^\t
$$
we can express the previous modularity function as
$Q_\gamma(S) =  \uno_S^\t M_{\mathrm{RN}} \uno_S$.
Another approach has been proposed in \cite{multires-2} by Arenas, Fernandes and Gomez. The following matrix is suggested as an alternative to the original Newman--Girvan modularity:
$$
   M_{\mathrm{AFG}} = A + \gamma I - 
   \frac{(d+\gamma\uno)(d+\gamma\uno)^\t}{\gamma n + \vol G} ,
$$
where $\gamma\in\mathbbm{R}$ is the resolution parameter, $n$ is the number of nodes of the graph and $d$ is the degree vector as usual. Note that the matrix $A+\gamma I$ is the adjacency matrix of $G$ where a self-loop with weight $\gamma$ is added to each node, so that $M_{\mathrm{AFG}}$ is nothing but the
Newman--Girvan matrix of the graph updated by 
the added loops.

\subsection{Generalized modularity matrices and measures}

Motivated by the aforementioned definitions, we consider the following generalization of 
the Newman--Girvan modularity matrix:

\begin{definition}\label{def:gen-mod}
Let $A$ be the the adjacency matrix of an undirected,
connected graph, possibly endowed by loops and weighted edges, let $\diag$ be a real diagonal matrix, let $v\neq 0$ be a nonnegative vector, 
and let $\sigma$
be a positive scalar. The matrix $M =  A + \diag - \sigma vv^T$
is a {\em generalized modularity matrix}. 
\end{definition}

According to Definition \ref{def:gen-mod}, it 
is clear that all previously defined modularity-type matrices 
$M_{\mathrm{NG}}$, $M_{\mathrm{norm}}$, $M_{\mathrm{RB}}$, $M_{\mathrm{AFG}}$ and $M_{\mathrm{RN}}$
are indeed generalized modularity matrices.

Hereafter, we adopt the notation $Q(S)$
to indicate the modularity measure of $S\subseteq V$ corresponding to (or induced by) a given generalized modularity matrix $M$,
that is,
$$
   Q(S) = \uno_S^\t M\uno_S\, .
$$
Remark that, if $G = (V,E)$ and 
$\diag = \mathrm{Diag}(w_1\ldots,w_n)$ then
the resulting expression for the modularity of $S$ is 
$$
   Q(S) = e_{\mathrm{in}}(S) + \sum_{i\in S}w_i
   - \sigma \bigg( \sum_{i\in S}v_i\bigg)^2  ,
$$
where $e_{\mathrm{in}}(S)$ is as in \eqref{eq:e_in}.
Thus, the diagonal matrix $\diag$ establishes 
a weight on each node; and the value of $Q(S)$ includes the sum
of all node weights in $S$. 

Finally, as it will play a crucial role in forthcoming discussions, we borrow from \cite{FT14} the notation
$$
   m_G = \lambda_1(M)
$$
to denote the leading 
(i.e., rightmost) eigenvalue of a generalized modularity matrix 
$M$ associated to the graph $G$. 
Owing to the inequality
$$
   m_G = \max_{x\neq 0} \frac{x^\t M x}{x^\t x} 
   \geq \frac{\uno_S^\t M \uno_S}{\uno_S^\t\uno_S}
   = \frac{Q(S)}{|S|} ,
$$
the existence of a module in $G$ implies that $m_G >0$.
Moreover, $m_G$ is an upper bound for the 
``relative modularity'' $Q(S)/|S|$.

Furthermore, for any two disjoint subsets $S,T\subseteq V$ 
we will consider their
{\em joint modularity} $Q(S,T) = \uno_S^\t M \uno_T$.
Note that $Q(S\cup T) = Q(S) + Q(T) + 2Q(S,T)$.
In particular, $Q(S\cup T) \geq Q(S) + Q(T)$
if and only if $Q(S,T) \geq 0$.

\section{Nodal domains of leading eigenvectors}\label{sec:main}

Given a nonzero vector $v\in\mathbbm{R}^n$ the 
subgraph $G(S)$ induced by the 
set $S = \{i : v_i \geq 0\}$ is a {\em nodal domain} of $v$
\cite{nodal-domain-theorem,duval-nodal-domains}.
This fundamental definition admits obvious variations
(for example, inequality can be strict, or reversed)
and, since the seminal papers by Fiedler \cite{fiedler-connectivity,fiedler-vector},
it has become the a major tool of 
spectral methods in community detection and graph partitioning
\cite{newman-eigenvectors,powers-graph-eigenvector,Schaeffer2007}.
Indeed, nodal domains of eigenvectors
of Laplacian or modularity matrices are commonly utilized 
in order to localize subgraphs having sought properties.

In this section
we consider nodal domains of the leading eigenvector
of generalized modularity matrices. In particular, the forthcoming 
Theorem \ref{thm:main} is 
a modularity matrix counterpart of Fiedler's theorem \cite[Thm. 3.3]{fiedler-vector} 
about Laplacian matrices. 

\begin{lemma}\label{lem:-1}
Let $A\geq O$ be irreducible and let $\diag$ be any real diagonal matrix. Then $\lambda_1(A+\diag)$ is simple and admits a 
positive eigenvector.
\end{lemma}

\begin{proof}
As $\diag$ is a real diagonal matrix, there exists a nonnegative scalar $\alpha$ such that the shifted matrix $A+ \diag + \alpha I$ is nonnegative and irreducible. Then the Perron--Frobenius theorem implies the thesis. 
\end{proof}

\begin{lemma}   \label{lem:0}
Let $M = A + \diag - \sigma vv^\t$ be a generalized modularity matrix. 
Then $m_G < \lambda_1(A+\diag)$. 
\end{lemma}

\begin{proof}
Weyl's inequalities \eqref{eq:Weyl} give
$m_G \leq \lambda_1(A+\diag)$. 
Suppose by contradiction $m_G =\lambda_1(A+\diag)$. 
Let $x$ and $y$ be eigenvectors corresponding to $m_G$ and 
$\lambda_1(A+\diag)$, that is,
$Mx = m_G x$ and $(A+\diag) y = \lambda_1(A + \diag)y$. 
By Lemma \ref{lem:-1} we can suppose $y > 0$. Hence,
$$
   m_G x^\t y = x^\t My = x^\t  
   (A + \diag -\sigma vv^\t )y =
   \lambda_1(A+\diag) x^\t y - \sigma(x^\t v)(v^\t y) .
$$
Thus $\sigma(x^\t v)(v^\t y) = 0$.
Since $\sigma(v^\t y) > 0$ we must have $x^\t v = 0$. Then,
$\lambda_1(A+\diag) x = m_G x = Mx = (A + \diag -\sigma vv^\t )x = (A+ \diag) x$.
Consequently, $x$ is an eigenvector of $A + \diag$  
corresponding to its first eigenvalue. Lemma \ref{lem:-1} implies either $x>0$ or $x<0$. In both cases 
$x^\t v = 0$ cannot hold.
\end{proof}

Interlacing properties between the spectra of $A+\diag$ and $M$ lead us immediately to the inequalities
$$
   \lambda_2(A+\diag) \leq m_G \leq \lambda_1(A+\diag) .
$$
The previous lemma shows that the rightmost inequality is always strict.
The next statement clarifies that, under common circumstances,
also the leftmost inequality is strict and $m_G$ is 
a simple eigenvalue of $M$.

\begin{theorem}
If $M = A + \diag - \sigma vv^\t$ is a generalized modularity matrix and $v$ is not
a leading eigenvector of $A+\diag$ then $m_G$
is a simple eigenvalue.
\end{theorem}

\begin{proof}
For an arbitrary vector $x$ we have
$x^\t M x = x^\t (A+\diag) x - \sigma(v^\t x)^2$. 
From Courant--Fisher's minimax theorem,
\begin{align*}
   m_G & = \max_{x\neq 0}\frac{x^\t M x}{x^\t x}
   \geq \max_{v^\t x= 0}\frac{x^\t (A+\diag) x}{x^\t x} \\
   & \geq \min_{z\neq 0}\max_{z^\t x= 0}\frac{x^\t (A+\diag) x}{x^\t x}
   = \lambda_2(A+\diag) .
\end{align*}
Thus we may have $m_G = \lambda_2(A+\diag)$ only if
the two preceding inequalities hold as equalities, that is,
$v^\t x = 0$ where $x$ is an eigenvector of $M$ associated to 
$m_G$, and $(A+\diag)v = \lambda_1(A+\diag)v$,
owing to orthogonality of eigenvectors of a symmetric matrix.
However, if the latter equation is verified, then
$v$ is also an eigenvector of $M$. Indeed,
$$
   \lambda_1 (A+\diag)v =  
   (A+\diag)v = Mv + \sigma(v^\t v) v ,
$$
whence $Mv = (\lambda_1(A+\diag) - \sigma v^\t v) v$.
Consequently, the equation $v^\t x = 0$ is redundant, again owing
to orthogonality of eigenvectors.
Finally, if $m_G$ is not simple then
$M$ has at least two eigenvalues strictly greater than $\lambda_2(A+\diag)$, which contradicts Weyl's inequalities \eqref{eq:Weyl},
and the proof is complete.
\end{proof}

\begin{lemma}   \label{lem:1}
Let $M = A + \diag - \sigma vv^\t$ be a generalized modularity matrix.
Let $Mx = m_G x$ 
and the eigenvector $x$ oriented so that $v^\t x \ge 0$.
Then, $S = \{ i: x_i \ge 0\}$ induces a connected subgraph.
\end{lemma}

\begin{proof}
By hypotheses, we have the componentwise inequality
$m_G x = M x = (A + \diag)x - (\sigma v^\t x) v \leq (A+\diag) x$.

By contradiction, assume that $S$ induces $2$ disjoint
connected subgraphs,
say $G(S_1)$ and $G(S_2)$. 
Reorder and partition consistently $A$, $\diag$, $M$, and $v$
in such a way that $S_1 = \{1,\ldots,n_1\}$, and
$S_2 = \{n_1+1,\ldots,n_2\}$. 
Consider the first $n_2$ equations in the inequality
$m_G x \leq (A+\diag) x$:
\begin{align*}
	\begin{pmatrix} m_G x_1 \\ m_G x_2 \end{pmatrix}
	& \leq \begin{pmatrix}
	A_{11}+\diag_{11} & & A_{13} \\ 
	& A_{22}+\diag_{22} & A_{23} \end{pmatrix}
	\begin{pmatrix} x_1 \\ x_2 \\ x_3 \end{pmatrix}
	= \begin{pmatrix} 
	A_{11}x_1 + \diag_{11}x_1 + A_{13} x_3 \\ 
	A_{22}x_2 + \diag_{22}x_2 + A_{23} x_3  
	\end{pmatrix} 
\end{align*}
Note that 
$x_3 < 0$ and $A_{i3} \neq O$ by irreducibility, for $i=1, 2$. In particular, we have strict 
inequality in at least one entry both in $S_1$ and in $S_2$.
 Let $y_1$ and $y_2$ be left eigenvectors of $A_{11}+\diag_{11}$ and
$A_{22}+\diag_{22}$, respectively such that:
$y_i^\t  (A_{ii} + \diag_{ii})= \lambda_1(A_{ii}+\diag_{ii}) y_i^\t $ for $i = 1,2$. Then,
$$
    m_G y_i^\t  x_i \leq 
    y_i^\t (A_{ii}+\diag_{ii})x_i + y_i^\t A_{i3}x_3
    < y_i^\t  (A_{ii}+\diag_{ii})x_i
    = \lambda_1(A_{ii}+\diag_{ii}) y_i^\t  x_i , 
$$
for $i = 1, 2$. Obviously, $y_i^\t x_i \geq 0$ 
since $y_i > 0$ by Lemma \ref{lem:-1} 
and $x_i \geq 0$ by hypothesis. Actually, due to the strict 
inequality above, we must have $y_i^\t x_i > 0$.
Thus both $\lambda_1(A_{11}+\diag_{11}) > m_G$ and $\lambda_1(A_{22}+\diag_{22}) > m_G$.
By eigenvalue interlacing inequalities \eqref{eq:interlacing2}, we conclude that 
$A+\diag$ has at least $2$ eigenvalues strictly larger than $m_G$, which contradicts Weyl's inequalities \eqref{eq:Weyl}.
\end{proof}

We can continue the argument in the previous proof as
follows. Let $y$ be any vector such that $(A + \diag)y \geq m_G y$.
For example, $y$ can be a positive eigenvector of $\lambda_1(A+\diag)$ since by Lemma \ref{lem:0} we know that $\lambda_1(A+\diag) > m_G$. 
Let $x+y = z$.
Thus $m_Gz \leq (A + \diag)z$ and, with arguments analogous to the ones exploited before, we obtain the following result.

\begin{theorem}\label{thm:main}
In the same hypotheses and notations of Lemma \ref{lem:1},
let $y$ be a positive eigenvector of $A+\diag$ 
corresponding to $\lambda_1(A+\diag)$.
Then, for any $\varepsilon \geq 0$,  
the set $S = \{i: x_i + \varepsilon y_i \geq 0\}$ 
induces a connected subgraph.
\end{theorem}

\begin{remark}   \label{rem:star}
A connectedness result concerning the set $S = \{ i: x_i > 0\}$
where $x$ is an eigenvector as in the hypotheses of Lemma \ref{lem:1}
can be obtained only under 
the additional assumption that $m_G$ is simple.
Indeed, consider the following example:
Let $G$ be a star graph on $n = m+1$ nodes,
with every node endowed by a loop carrying the weight $\sqrt{m}$.
Its adjacency matrix is
$$
   A = \begin{pmatrix}
   \sqrt{m} & 1 & \cdots & 1 \\ 1 & \sqrt{m} \\
   \vdots & & \ddots \\ 1 & & & \sqrt{m} \end{pmatrix} .
$$
Easy computations show that $M_{\mathrm{NG}}$
has an $(m-1)$-fold leading eigenvalue $m_G$ equal to $\sqrt{m}$; 
every associated eigenvector is a zero-sum vector vanishing at the star center.
Consequently, if $x$ is a leading eigenvector of $M_{\mathrm{NG}}$
then $S = \{ i: x_i > 0\}$ is connected if and only if it reduces to a single node.
We will not pursue here this argument, and point the interested reader 
to Section 4 of \cite{FT14}. 
\end{remark}

\subsection*{An applications to community detection}

In this subsection we describe a major application of Theorem \ref{thm:main} to the community detection problem through the following Corollary \ref{cor:SSGB}. First of all let us underline that as soon as the modularity measure is induced by generalized modularity matrix $M$, it is reasonable to consider the SSGB procedure (see Section \ref{sec:motivations}) applied to $M$, in order to subdivide the graph into modules. However the definition of the matrix in \eqref{eq:sub-modularity} is not always well posed. The matrix $M_{\mathrm{NG}}^S$ therein considered is the 
Newman--Girvan modularity matrix associated to the subgraph $G(S)$ 
induced by $S$. 
However, the structure of a generalized modularity matrix $M=A+\diag -\sigma vv^\t$ might be only partially defined in terms of  $G$, as $\diag$, $v$ and $\sigma$ may be arbitrary.  Let us agree now that, if this is the case, then we 
denote with $M^{S}$ the principal submatrix of $M$ whose indices are in $S$.
Otherwise let $M^{S}$ denote the generalized modularity matrix defined in terms of $G(S)$. With this  notation the SSGB scheme  survives unchanged when $M_{\mathrm{NG}}$ is replaced by a generic $M$. 
The next corollary shows that Theorem \ref{thm:main} gives us informations on the connectivity of the modules produced by the SSGB method, whenever the modularity measure is induced by a generic $M$.

\begin{corollary}\label{cor:SSGB}
Let $M$ be any generalized modularity matrix. The spectral method applied to $M$ generates a pair of subgraphs, one of which is certainly connected. Similarly, the SSGB method applied to $M$ generates $m$ subgraphs, half of which is connected.
\end{corollary}

\begin{proof}
Let $M = A+\diag-\sigma vv^\t$. It is enough to observe that,  if $u$ is the eigenvector corresponding to the larger eigenvalue $m_G$ oriented so that $u^\t v \geq 0$ then, due to Theorem \ref{thm:main}, the set $S=\{i : u_i \geq 0\}$ 
defines a bipartition of the node set such that $G(S)$ is connected. 
However, we have no apriori control on the connectivity of the other set of the bipartition. 

A similar argument proves the thesis for the SSGB scheme. Due to the successive bipartitions, the algorithm produces  $m/2$ pairs of subsets. To be precise, at each step of the scheme, a subset $S\subseteq V$ is given, then the matrix $M^S$ 
is computed and $S$ is partitioned into a pair of subsets being identified by the sign of the entries of the leading eigenvector of $M^S$.   
Note that, for any generalized modularity matrix $M$, the matrix $M^S$ has the form $M^S = A' + \diag' + R$, where $A'$ is a nonnegative symmetric matrix, $\diag'$ is diagonal and $R$ is a negative definite rank one matrix. 
Observe now that, if $u_S$ is an eigenvector corresponding to the largest eigenvalue of $M^S$, and its sign is chosen 
appropriately,
then the hypothesis of  Theorem \ref{thm:main} are satisfied. 
As a consequence the set $\tilde S =\{i \in S : (u_S)_i \geq 0\}$ induces a connected subgraph in $G(S)$, thus a connected subgraph in $G$.  
\end{proof}

\section{A criterion for the leading eigenpair} 

The classical Perron--Frobenius theory has been extended in various ways to matrices having some negative entries. 
One of such extensions, found in \cite{MR1823511},
allows us to predict the sign pattern in the leading eigenvector of a generalized modularity matrix.

\begin{lemma}   \label{lem:MR1823511}
Let $P\in\mathbbm{R}^{n\times n}$ be a symmetric matrix.
If $\uno^\t P\uno \geq \sqrt{(n-1)^2+1} \|P\|_\fro$ then
$\rho(P)$ is an eigenvalue of $P$ which is simple
and associated to a nonnegative eigenvector.
\end{lemma}

\begin{proof}
See \cite[Thm.\ 4.1]{MR1823511}.
\end{proof}

Remark that the preceding lemma makes no assumptions
on signs and sizes of the entries of the matrix $P$.
In fact, various examples shown in \cite{MR1823511}
illustrate that the hypotheses of this lemma
can be fulfilled by matrices having some negative entries.
 
\begin{theorem}
Let $M\in\mathbbm{R}^{n\times n}$ be a 
generalized modularity matrix.
Let $S\subset V$ be a set fulfilling the inequality
$$
   Q(S) + Q(\bar S) - 2Q(S,\bar S) \geq
   \sqrt{(n-1)^2+1} \|M\|_\fro .
$$
Then $\rho(M) = m_G$ is a simple eigenvalue of $M$ which is 
associated to an eigenvector $x$ with the 
following property:
$S = \{i: x_i \geq 0\}$.
\end{theorem}

\begin{proof}
Let $J$ be the diagonal matrix such that 
$J_{ii} = 1$ if $i\in S$ and $J_{ii} = -1$
otherwise.
Moreover, let $P = J M J$.
Observe that 
$$
   \uno^\t P \uno = 
   (\uno_S - \uno_{\bar S})^\t M (\uno_S - \uno_{\bar S})
   = Q(S) + Q(\bar S) - 2Q(S,\bar S) .
$$
On the other hand, $\|P\|_\fro = \|M\|_\fro$. Finally, 
$x$ is an eigenvector of $M$ if and only if $J x$
is an eigenvector of $P$. Thus the claim
follows from Lemma \ref{lem:MR1823511}.
\end{proof}

However, since for a modularity matrix $M$
the rightmost eigenvalue may not be equal to the spectral radius,
it is useful to derive a weakened version of the 
previous theorem which considers the matrix pencil $M + \alpha I$.

\begin{corollary}
Let $M\in\mathbbm{R}^{n\times n}$ be a 
generalized modularity matrix.
Let $S\subseteq V$ be a set fulfilling the inequality
$$
   Q(S) + Q(\bar S) - 2Q(S,\bar S) \geq
   \sqrt{(n-1)^2+1} \|M + \alpha I\|_\fro - n\alpha
$$
for some $\alpha \in\mathbbm{R}$.
Then the rightmost eigenvalue of $M$ is simple
and associated to an eigenvector $x$ such that 
$S = \{i: x_i \geq 0\}$.
\end{corollary}

\begin{proof}
Repeat the argument in the previous proof
with the matrix $M$ replaced by $M + \alpha I$.
Note that, in this case, 
$\uno^\t P \uno = Q(S) + Q(\bar S) - 2Q(S,\bar S) + n\alpha$.
\end{proof}

Observe that the effect of introducing the shift $M+\alpha I$  is twofold:
If $	\alpha >0$ then the spectrum of $M$ is translated to the right, and the rightmost eigenvalue may become the spectral radius of the shifted matrix. Moreover, the Frobenius norm of the shifted matrix may be smaller than that of $M$, for example,
when the graph has no loops; in that case,
the diagonal of $M$ is negative and $\|M\|_\fro$ can be decreased
by means of a small positive shift.
Indeed, note that $\|M+\alpha I\|_\fro$ is minimum when 
$\alpha = -\mathrm{trace}(M)/n$.

\section{Sensibility under small perturbations}

Let $G = (V,E)$ be a given graph and let $S\subseteq V$
be a set 
with $Q(S) > 0$ where $Q$ is the modularity function induced by
$M_{\mathrm{NG}}$.
Let $(i,j)$ be an edge missing in $G$, and let $G' = (V,E')$
be the graph obtained by adding to $G$ that edge:
$E'(i,j) > 0$. It is not difficult to verify that,
due to the new edge,
\begin{itemize}
\item
if both $i\in S$ and $j\in S$ then $Q(S)$ increases,
\item
if $i\in S$ and $j\notin S$ then $Q(S)$ decreases.
\end{itemize}
Analogous behaviours can be observed by using other modularity-type functions, among those recalled in Section \ref{sec:motivations}.
Indeed, in some sense, in the first case the new edge increases the internal connection, and $S$ becomes a stronger community than before; while in the latter case $S$ becomes less separated from its exterior, hence it is less recognizable as a community.
It is natural to ask whether that ``monotonicity property'' of the modularity function is somewhat preserved by $m_G$.
Indeed, from our standpoint,
it makes sense to observe the variation of the rightmost eigenvalue $m_G$ of $M$ 
after a small increment on the weight of one of the edges of $G$.
Accordingly, in place of the conditions like $i\in S$ or $i\notin S$, we consider the sign of the $i$-th entry of a corresponding eigenvector of $M$. In fact, nodal domain based methods 
employ signs of eigenvector entries to locate possible communities: a positive value indicates that the vertex belongs to a cluster and a negative value that it is outside the cluster.

Of course if $M=A + \diag -\sigma vv^\t$  is a generalized modularity matrix for $G$, the modularity matrix $M' = A' + \diag'-\sigma^\prime v^\prime v^{\prime \t}$ for the new graph $G'$ should be defined properly. Although it is clear what $A'$ is, the matrices $\diag'$ and $\sigma^\prime v^\prime v^{\prime\t}$ may have not a clear definition.  
For definiteness, we consider the following assumption:
If the graph $G$ is perturbed by adding a weight $\ep > 0$
to the edge $(i,j)$ then
$\diag^\prime = \diag$ and there exists a symmetric matrix $E$ such that 
$\sigma^\prime v^\prime v^{\prime \t} = \sigma (I+E)vv^\t(I+E)$
and $\|E\|_2 \leq \eta\ep$ for some $\eta$.
That is, we assume that the rank-one term in $M'$ is a small
relative perturbation of that in $M$.
That assumption is fulfilled in practice by 
all modularity-type matrices introduced in Section \ref{sec:motivations}.
Note that it is possible to consider as $E$ the diagonal matrix
whose diagonal entries are
$$
   E_{ii} =
   \frac{\sqrt{\sigma'}v'_i - \sqrt{\sigma}v_i}{\sqrt{\sigma}v_i}  
   \qquad i = 1,\ldots, n .
$$
A possible result along this direction is discussed throughout the remaining part of this section.

\begin{definition}
Let $G_0$ be a given graph, let $i,j\in V$ be a fixed pair of vertices, and let $G_\varepsilon$ be the graph obtained by adding the edge in $(i,j)$ to $G_0$ with weight $\varepsilon>0$. (Assume that,
if $(i,j)$ 
is an edge in $G_0$ then its weight in $G_\varepsilon$ is increased by $\ep$.)
Let $M_0$ and $M_\ep$ be generalized modularity matrices
of $G_0$ and $G_\ep$, respectively. 
Define
$$
   \mu_{ij}^\ep = \frac{m_{G_\ep} - m_{G_0}}{\ep} \, .
$$
\end{definition}

Let $M_0$ and $M_\ep$ as in the previous definition.
If $M_\ep - M_0 = \varepsilon(e_ie_j^\t +e_je_i^\t)$
and $\lambda_1(M_0)$ is simple
then $\lambda_1(M_\ep)$ varies according to the sign
of $x_ix_j$, where $x$ is a leading eigenvector of $M_0$,
at least for sufficiently small $\ep$. 
In fact, from classical results in eigenvalue perturbation theory
\cite{wilkinson}, in the stated hypotheses
$\lambda_1(M_\ep)$ is differentiable for small $\varepsilon$,
whence $\mu_{ij}^\ep = \lambda'_1(M_0) + o(\ep)$.
Moreover, assuming that $x$ is normalized, we have
$$
   \lambda'_1(M_0) = \ep^{-1}x^\t (M_\ep - M_0) x
   = x^\t (e_ie_j^\t +e_je_i^\t) x =
   2 x_ix_j ,
$$
showing indeed that $\mu_{ij}^\ep = 2x_ix_j + o(\ep)$.
Now consider the general case where, according to our previous assumption, we have
$$
   M_\ep - M_0 = \varepsilon(e_ie_j^\t +e_je_i^\t) -
   \sigma [(I+E)vv^\t(I+E) - vv^\t] .
$$
Then,
\begin{align*}
   x^\t (M_\ep - M_0)x & = 2\ep x_ix_j - 
   \sigma [x^\t (I+E)vv^\t (I+E)x - x^\t vv^\t x]   \\
   & = 2\ep x_ix_j - 
   \sigma [(x^\t (I+E)v)^2 - (v^\t x)^2]   \\
   & = 2\ep x_ix_j - 
   \sigma (x^\t(2I+E)v)(v^\t Ex) .   
\end{align*}   
Let $\cos\theta = (x^\t v)/\|v\|_2\|x\|_2$ be the cosine of the angle
between $x$ and $v$.
Taking norms and assuming $\|x\|_2 = 1$ as before, we obtain
$$
   |\mu_{ij}^\ep - 2x_ix_j | \lesssim 2\eta|\cos\theta|\,\sigma \|v\|_2^2 
   = 2\eta|\cos\theta|\, \|\sigma vv^\t\|_2 ,
$$
neglecting lower order terms.
Thus, if $\eta$ and $|\cos\theta|$ are sufficiently small then 
$\mu_{ij}^\ep$ has the same sign of $x_ix_j$.
In particular, if the new edge is added between two nodes 
having the same sign in $x$ then the algebraic modularity increases,
and conversely, if $x_ix_j < 0$.

\section{Positive eigenvalues and number of modules}

On the basis of rather informal arguments, Newman claims in 
\cite[Sect.\ B]{newman-eigenvectors} that the number of positive eigenvalues of $M_{\mathrm{NG}}$ is related to the number of communities recognizable in the graph $G$. 
The subsequent Theorem \ref{thm:k}, which generalizes an analogous result
concerning the matrix $M_{\mathrm{NG}}$ shown in \cite[Thm.\ 6.2]{FT14},
proves that for any generalized modularity matrix and
the modularity function $Q$ associated to it, 
the number of positive eigenvalues of $M$ is actually an upper bound for the cardinality of any family of pairwise disjoint modules in $G$
having the property that, if any two modules are merged then the overall modularity does not increase.

\begin{lemma}   \label{lem:k}
Let $S_1,\ldots,S_k$ be $k$ pairwise disjoint, nontrivial subsets of $V$,
with $k \ge 1$. Let $C$ be the $k\times k$ symmetric matrix with 
$C_{ij} = \uno_{S_i}^\t M\uno_{S_j}$
where $M$ is any modularity matrix. 
The number of positive (nonnegative) eigenvalues 
of $M$ is not smaller than the number of 
positive (nonnegative, respectively) eigenvalues 
of $C$.
\end{lemma}

\begin{proof}
Consider the matrices $Z = [\uno_{S_1} \cdots \uno_{S_k}]$
and $\mathit \Sigma = \mathrm{Diag}(|S_1|,\ldots,|S_k|)^{-1/2}$.
Note that $\hat Z = Z\mathit\Sigma$ has orthonormal columns.
By Sylvester's law of inertia, the number of positive 
(nonnegative) eigenvalues of $C$
coincides with the number of positive (nonnegative, respectively) eigenvalues of
$\mathit\Sigma C \mathit\Sigma = \hat Z^\t M \hat Z$.
The claim follows by Cauchy interlacing inequalities \eqref{eq:interlacing}.
\end{proof}

Given a family of pairwise disjoint subsets 
$\mathcal{P} = \{S_1,\ldots,S_k\}$ one usually defines the 
modularity of $\mathcal{P}$ as $Q(\mathcal{P}) = \sum_iQ(S_i)$.
The maximization of the latter quantity is 
a recurrent task in community detection algorithms 
\cite{newman-eigenvectors,Schaeffer2007,TVDN_CPM}.
If each $S_i$ is a module, $\mathcal{P}$
maximizes $Q(\mathcal{P})$
and contains the least number of sets among all such families,
then $Q(S_i,S_j) < 0$ for $i\neq j$,
otherwise we can reduce $|\mathcal{P}|$ or increase $Q(\mathcal{P})$ (or both) by merging 
subsets whose joint modularity is nonnegative. 
In that case, the matrix $C$ introduced in the preceding lemma
has a sign pattern which is well known in the field of nonnegative matrices \cite{positive-book}. One possible consequence is 
stated in the forthcoming result, relating the number
of positive eigenvalues of $M$ to the number of disjoint modules in $G$ that optimize the overall modularity.

\begin{theorem}   \label{thm:k}
Let $S_1,\ldots,S_k$ be $k$ pairwise disjoint, nontrivial subsets of $V$,
with $k \ge 1$. Suppose that, for all $i = 1,\ldots,k$,
we have $Q(S_i) > 0$ and $Q(S_i,S_j)<0$ for $i\neq j$. 
If there exist positive numbers $\alpha_1,\dots, \alpha_k$ such that 
$$
   \alpha_i Q(S_i) > \sum_{j\neq i} \alpha_j |Q(S_i,S_j)|\, .
$$
then $M$ has at least $k$ positive eigenvalues.
\end{theorem}
\begin{proof}
Let $C$ be the $k\times k$ symmetric matrix with 
$C_{ij} = \uno_{S_i}^\t M\uno_{S_j}$ and let 
$\alpha = (\alpha_1,\ldots, \alpha_k)^\t$.
In the stated hypotheses $C_{ii} > 0$,
$C_{ij} < 0$ for $i\neq j$ and $C\alpha > 0$.
By a classical result on nonnegative matrices \cite[\S 6.2]{positive-book}
$C$ is a symmetric M-matrix, so in particular it is positive definite.
The claim follows immediately from Lemma \ref{lem:k}.
\end{proof}

It is worth noting that the condition on $\alpha_1,\dots, \alpha_k$ 
in the previous theorem can be easily fulfilled when $S_1,\ldots,S_k$
is a partition of $V$ and $M = M_{\mathrm{NG}}$
or $M = M_{\mathrm{AFG}}$, see Section \ref{sec:motivations}.
Indeed, in those cases we have $M\uno = 0$
and, consequently, we can obtain the sought inequalities by setting $\alpha_1 = \ldots = \alpha_k = 1$ as shown in the forthcoming corollary. 

\begin{corollary}
Let $\mathcal{P} = \{S_1,\ldots,S_p\}$ be a partition of $V$
into pairwise disjoint subsets. 
Suppose that $Q(S_i) > 0$ and $Q(S_i,S_j) < 0$ for $i\neq j$, 
where $Q$ is the modularity function associated to a generalized modularity 
matrix $M$ such that $M\uno = 0$.
Then the number of positive eigenvalues of $M$ is at least $p-1$.
\end{corollary}

\begin{proof}
Let $Z = [\uno_{S_1}\cdots \uno_{S_p}]$ and $C = Z^\t MZ \equiv (\uno_{S_i}^\t M\uno_{S_j})$. Since $Z\uno = \uno$,
by hypothesis we obtain $C \uno = Z^\t MZ\uno = 0$. Then,
for $i = 1,\ldots,p-1$ we have
$$
   Q(S_i) + \sum_{j\neq i,p} Q(S_i,S_j) = -Q(S_i,S_p) > 0 .
$$
Using Theorem \ref{thm:k} with $k = p-1$ we obtain the claim.
\end{proof}

We close this section with the forthcoming theorem which states that, 
if $G$ has $k$ subgraphs that are well separated and sufficiently rich in internal edges (including loops), then
$M$ has at least $k-1$ positive eigenvalues. 
This result extends Theorem 6.1
in \cite{FT14} to arbitrary generalized modularity matrices.
For better clarity consider that,
if $S$ and $T$ are two disjoint subsets of $V$, then the number $\uno_{S}^\t A\uno_{T}$ corresponds to
the total weight of edges joining nodes in $S$ with nodes in $T$.
\begin{theorem}
Let $S_1,\ldots,S_k$ be pairwise disjoint subsets of $V$, 
with $k \geq 1$, such that 
$$
   e_{\mathrm{in}}(S_i) + \uno_{S_i}^\t \diag\uno_{S_i}
   > \sum_{j\neq i} \uno_{S_i}^\t A\uno_{S_j} .
$$
Then $M$ has at least $k-1$ positive eigenvalues.
\end{theorem}
\begin{proof}
Consider the matrices $Z$ and $\mathit\Sigma$ introduced in
the proof of Lemma \ref{lem:k}.
Introduce the $k\times k$ matrix $B = Z^\t (A+\diag)Z$.
We have $B_{ii} = e_{\mathrm{in}}(S_i) + \uno_{S_i}^\t \diag\uno_{S_i}$ and
$B_{ij} = \uno_{S_i}^\t A\uno_{S_j}$ for $i \neq j$. 
By hypothesis, $B$ is nonnegative and strictly diagonally dominant,
hence it is positive definite. 
Consider the matrix $C$ defined in Lemma \ref{lem:k}:
$$
   C = Z^\t M Z = Z^\t (A+\diag - \sigma vv^\t) Z = 
   B - \sigma (Zv)(Zv)^\t .
$$
We see that $C$ is a negative semidefinite,
rank-one perturbation of $B$, hence it has at least $k-1$ positive eigenvalues.
The claim follows from Lemma \ref{lem:k}.
\end{proof}

\section{Conclusions}
\label{sec:conclusions}

Community detection is a major problem arising in modern complex network analysis, and modularity-type matrices and functions play a fundamental role in network science. In fact, several generalizations of the modularity matrix originally introduced 
by Newman and Girvan 
\cite{newman-eigenvectors,newman-modularity,newman-girvan} appear in the complex networks literature, often in a rather hidden form. In this paper we put in evidence 
that a common structure and various spectral properties are shared by all these matrices.
As the matrix theoretic approach to modularity based methods is very recent,
several directions of investigation are left open. Relevant steps would be, in our opinion, to provide lower bounds for the modularity of graphs in terms of the spectrum of the associated modularity matrix, and robustness results of leading modules with respect to different modularity measures.



\bibliography{./networks_ref_jan15}
\bibliographystyle{plain}

\end{document}